\documentclass[a4paper,12pt,reqno]{amsart}
 
\usepackage[margin=2cm]{geometry} 
\usepackage{amsmath,amsthm,amssymb,enumitem}
\usepackage{amsrefs}
\usepackage[ocgcolorlinks,hyperfootnotes=false,colorlinks=true,citecolor=blue,linkcolor=blue,urlcolor=blue]{hyperref}

\newcommand{\fl}{f{\kern0.075em}l}
\newcommand{\norm}[1]{\left\lVert #1 \right\rVert}

\theoremstyle{plain}
\newtheorem{thm}{Theorem}[section]
\newtheorem{lem}[thm]{Lemma}
\newtheorem{prop}[thm]{Proposition}

\theoremstyle{definition}
\newtheorem{defn}[thm]{Definition}
\newtheorem{exmp}[thm]{Example}
\newtheorem{manualtheoreminner}{Theorem}
\newenvironment{manualtheorem}[1]{%
  \renewcommand\themanualtheoreminner{#1}%
  \manualtheoreminner
}{\endmanualtheoreminner}

\setlist[enumerate]{itemsep=2mm, topsep=0mm}

\title[Boundary behaviour of the Bergman and Szeg\H{o} kernels]{Boundary behaviour of the Bergman and Szeg\H{o} kernels on generalized decoupled domains
}
\author{Ravi Shankar Jaiswal}
\address{Centre for Applicable Mathematics, Tata Institute of Fundamental Research, Bangalore 560065, India.}
\email{ravi@tifrbng.res.in}
\date{December 20, 2023}

\begin{document}

\addtolength{\jot}{2mm}
\addtolength{\abovedisplayskip}{1mm}
\addtolength{\belowdisplayskip}{1mm}

\maketitle
\begin{abstract} 
We prove optimal estimates of the Bergman and Szeg\H{o} kernels on the diagonal, and the Bergman metric near the boundary of bounded smooth generalized decoupled pseudoconvex domains in $\mathbb{C}^n$. 
The generalized decoupled domains we consider allow the following possibilities: (a) complex tangential directions need not be decoupled separately, and (b) boundary points could have both finite and infinite type directions.
\end{abstract}
\section{Introduction}\label{Intro} 
The Bergman and Szeg\H{o} kernels hold significance as fundamental reproducing kernels in complex analysis, each with its distinct properties. The Bergman kernel, as a measure, is biholomorphically invariant, whereas the Szeg\H{o} kernel does not possess this property. The Bergman kernel is closely associated with the $\overline{\partial}$-problem and the $\overline{\partial}$-Neumann Laplacian, while the Szeg\H{o} kernel is related to the $\overline{\partial}_b$-problem and the Kohn Laplacian.

McNeal \cite{McNeal 1991} studied the boundary behaviour of the Bergman kernel and metric near finite type boundary points of decoupled domains in $\mathbb{C}^n$. Bharali \cites{Bharali 2010, Bharali 2020} proved optimal estimates of the Bergman kernel and metric near exponentially flat infinite type boundary points of domains in $\mathbb{C}^2$. 

We work on a class of domains that generalize the decoupled domains considered by McNeal \cite{McNeal 1991}. In our class of domains, complex tangential directions need not be decoupled individually and boundary points could have both finite and infinite type directions.
Here is an example of such a domain:
\begin{equation}\label{example-1}
    \left\{(z_1, z_2, z_3, z_4) \in \mathbb{C}^4:\operatorname{Re}z_4 > e^{-1/\left(|z_1|^2 + |z_2|^2\right)} + |z_3|^2\right\}.
\end{equation}

Note that the complex tangential directions at the origin are not decoupled separately, and the boundary is finite type in the $z_3$ direction while being infinite type in the $z_1$ and $z_2$ directions.

The primary goal of this article is to prove the optimal lower and upper bounds of the Bergman and Szeg\H{o} kernels on the diagonal, and the Bergman metric on this class of generalized decoupled domains.
We now define generalized decoupled domains.
\begin{defn}
A bounded smooth domain $D \subset \mathbb{C}^n$ with $0 \in bD$ is said to be a
\emph{generalized decoupled domain near $0$} if there exists a local defining function of $D$ near the origin of the form
\begin{align}
    \rho(z', z_n) = F(z') - \operatorname{Re}z_n,
\end{align}
where $F : \mathbb{C}^{n - 1} \to \mathbb{R}$ is a smooth function, satisfying
\begin{align}\label{eqF}
    F(z') = f_1\left(\left|z^1\right|\right) + \dots + f_k\left(\left|z^k\right|\right).
\end{align}
Here \(z' = \left(z^1, \dots, z^k\right) \in \mathbb{C}^{n_1} \times \dots \times \mathbb{C}^{n_k}\), 
where $k, n_1, \dots, n_k \in \mathbb{N}$ satisfy the condition $n_1 + \dots + n_k = n - 1$, and each \(f_j\) is either \emph{finite type at zero} or \emph{mildly infinite type at zero} (refer to Definitions \ref{finite type} and \ref{infinite type} for the precise definition of these terms).
\end{defn}
The following theorem gives optimal estimates of the Bergman kernel on the diagonal.
\begin{thm}\label{Bergman kernel}
    Let $D \subset \mathbb{C}^n$ be a bounded smooth pseudoconvex generalized decoupled domain near $0 \in bD$.
\begin{enumerate}
    \item Then, there exist constants $C > 0$ and $t_0 > 0$ such that
\begin{align}
    \kappa_{D}(z) \geq C (\operatorname{Re}z_n)^{-2} \Pi_{j = 1}^{k} [f_j^{-1}(\operatorname{Re} z_n)]^{-2n_{j}},
\end{align}
for each $z \in \{w \in \mathbb{C}^n : \operatorname{Re}w_n > F(w'), \text{ and }|w_n| < t_0\}$.
\item Then, for each $\alpha = (\alpha_1, \dots, \alpha_k)$ with ${\alpha_j > 1}$
$(j \in \{1, \dots, k\})$ and $\beta > 1$, there exist constants $C(\alpha, \beta) > 0$ and $t_0 > 0$ such that
\begin{align}
    \kappa_{D}(z) \leq C(\alpha, \beta) (\operatorname{Re}z_n)^{-2} \Pi_{j = 1}^{k} [f_j^{-1}(\operatorname{Re} z_n)]^{-2n_{j}},
\end{align}
for each $z \in \Gamma^{\alpha, \beta}_k \cap \{w \in \mathbb{C}^n : |w_n| < t_0\}$.
Here
\begin{align}
    \Gamma^{\alpha, \beta}_{k} := \left\{w \in \mathbb{C}^n: \operatorname{Re}w_n > k \beta \left[f_1\left(\alpha_1\left|w^1\right|\right) + \dots + f_k\left(\alpha_k\left|w^k\right|\right)\right]\right\}.
\end{align}
\end{enumerate}
\end{thm}
Our next theorem gives optimal estimates of the Bergman metric.
\begin{thm}\label{Berman metric}
Let $D \subset \mathbb {C}^n$ be a bounded smooth pseudoconvex generalized decoupled domain near $0 \in bD$.
Then, for each $\alpha = (\alpha_1, \dots, \alpha_k)$ with $\alpha_j > 1$ $(j \in \{1, 2, \dots, k\})$ and $\beta > 1$, there exist constants $C(\alpha, \beta) > 0$ and $t_0 > 0$ such that 
\begin{align*}
    \frac{1}{C(\alpha, \beta)}\left(\sum_{j = 1}^{k}\frac{|\xi^j|^2}{(f_j^{-1}(\operatorname{Re}z_n)^2} + \frac{|\xi_n|^2}{(\operatorname{Re}z_n)^2}\right)\leq B^2_D(z; \xi) \leq C(\alpha, \beta) \left(\sum_{j = 1}^{k}\frac{|\xi^j|^2}{(f_j^{-1}(\operatorname{Re}z_n)^2} + \frac{|\xi_n|^2}{(\operatorname{Re}z_n)^2}\right)
\end{align*}
for each $z \in \Gamma^{\alpha, \beta}_{k} \cap \{w \in \mathbb{C}^n: |w_n| < t_0\}$, and $\xi = (\xi^1, \dots, \xi^k, \xi_n) \in \mathbb{C}^{n_1} \times \dots \times \mathbb{C}^{n_k} \times \mathbb{C}$. 
\end{thm}
The following theorem gives optimal estimates of the Szeg\H{o} kernel on the diagonal.
\begin{thm}\label{Theorem 1.3}
Let $D \subset \mathbb{C}^n$ be a bounded smooth pseudoconvex generalized decoupled domain near $0 \in bD$.
\begin{enumerate}
    \item Then, there exist constants $C > 0$ and $t_0 > 0$ such that
    \begin{align}
    \mathcal{S}_D(z) \geq C (\operatorname{Re}z_n)^{-1} \Pi_{j = 1}^{k} [f_j^{-1}(\operatorname{Re} z_n)]^{-2n_{j}},
\end{align}
for each $z \in \{w \in \mathbb{C}^n: \operatorname{Re}w_n > F(w'), \text{ and } |w_n| < t_0\}$.
    \item Suppose $bD$ is convex near the origin.  Then, for  each $\alpha = (\alpha_1, \dots, \alpha_k)$ with $\alpha_j > 1$ $(j \in \{1, 2, \dots, k\})$ and $\beta > 1$, there exist constants $C(\alpha, \beta) > 0$ and $t_0 > 0$ such that
\begin{align}
    \mathcal{S}_{D}(z) \leq C({\alpha})(\operatorname{Re}z_n)^{-1} \Pi_{j = 1}^{k} [f_j^{-1}(\operatorname{Re} z_n)]^{-2n_{j}},
\end{align}
for each $z \in \Gamma^{\alpha, \beta}_k \cap \{w \in \mathbb {C}^n :|w_n| < t_0\}$. 
\end{enumerate}
\end{thm}
We now prove lower bounds on the Bergman and Szeg\H{o} kernels of domains that are more general than the ones considered above. These
domains are extended generalized decoupled domains where the decoupling need not necessarily respect the complex structure. 
The following is an example of such a domain:
\begin{align}
    \left\{(z_1, z_2, z_3) \in \mathbb{C}^3: \operatorname{Re}z_3 > \operatorname{exp}\left(-1/\left[(\operatorname{Re}z_1)^2 + (\operatorname{Im}z_2)^2\right]\right) + (\operatorname{Re}z_2)^4 + (\operatorname{Im}z_1)^2\right\}.
\end{align}
We now define this class of domains.
\begin{defn}\label{A}
A domain $D \subset \mathbb{C}^n$ with $0 \in bD$ is said to be an
\emph{extended generalized 
decoupled domain near $0$} if there 
exist a neighbourhood $U \subset 
\mathbb{C}^n$ of $0$ and a constant 
$\delta > 0$ such that
\begin{align}
    &D \cap U = \{z \in U : \operatorname{Re}z_n > F(z')\},  \quad D \subset \{z \in \mathbb{C}^n: \operatorname{Re}z_n > 0\}, \, \text{and} \label{condition1}\\
    &D \cap U^c \subset \{z \in D : |z_n| > \delta \}, \label{condition2}
\end{align}
where $F : \mathbb{C}^{n - 1} \to \mathbb{R}$ is a smooth function, satisfying
\begin{align}
    F(z') = f_1\left(\left|s^1\right|\right) + \dots + f_k\left(\left|s^k\right|\right). \label{eqF2}
\end{align}
Here \(z' = \left(s^1, s^2, \dots, s^k\right) \in \mathbb{R}^{n_1} \times \mathbb{R}^{n_2} \times \dots \times \mathbb{R}^{n_k}\), where $k, n_1, \dots, n_k \in \mathbb{N}$ satisfy the condition $n_1 + \dots + n_k = 2(n - 1)$, and each \(f_j\) is either \emph{finite type at zero} or \emph{mildly infinite type at zero}.

The conditions \eqref{condition1} and \eqref{condition2} are easily satisfied by the domain $\{z \in \mathbb{C}^n: \operatorname{Re}z_n > F(z') \}$, where $F$ is defined as in \eqref{eqF2}. 
\end{defn}
We utilize a similar approach to the one employed in proving Theorem \ref{Bergman kernel} and Theorem \ref{Theorem 1.3} to establish the subsequent theorem.
\begin{thm}\label{Theorem 1.4}
    Let $D \subset \mathbb{C}^n$ be a bounded $C^2$-smooth domain with $0 \in bD$. Suppose $D$ is an extended generalized decoupled domain near $0$.
    Then there exists constant $C(\delta) > 0$ such that
\begin{align}
    \kappa_D(z) &\geq C(\delta) (\operatorname{Re}z_n)^{-2} \Pi_{j = 1}^{k} [f_j^{-1}(\operatorname{Re} z_n)]^{-n_{j}}, \text{ and} \label{Bergman*}\\
    \mathcal{S}_{D}(z) &\geq C(\delta) (\operatorname{Re}z_n)^{-1} \Pi_{j = 1}^{k} [f_j^{-1}(\operatorname{Re} z_n)]^{-n_{j}},\label{Szegö*}
\end{align}
for each $z \in \{w \in D : |w_n| < \delta/2\}$.
\end{thm}
Numerous contributions have been made in the study of the boundary behaviour of Bergman and Szegő kernels, as well as the Bergman metric, across various classes of domains. 
McNeal \cites{McNeal 1989, McNeal 1994} derived estimates of the Bergman kernel on bounded convex domains of finite type in $\mathbb{C}^n$.
Additionally, he \cite{McNeal 1992} established the lower bound of the Bergman metric near points of finite type in 
bounded domains in $\mathbb{C}^n$. Other notable contributions 
include the works of Catlin \cite{Catlin}, Nagel, Rosay, Stein, and Wainger \cite{Nagel}, Herbort \cites{Herbort 1992, Herbort 1993}, and Diedrich \cites{Diedrich 1993, Diedrich 1994}. Some progress has also been made in estimating these objects on pseudoconvex domains of infinite type.
Nikolov, Pflug, and Zwonek \cites{Nikolov 2011} proved optimal estimates of the Bergman kernel and metric on $\mathbb{C}$-convex domains in $\mathbb{C}^n$ that do not contain complex lines.

Moreover, significant strides have been made in proving the boundary limits of the Bergman kernel and metric on weakly pseudoconvex domains of finite type and certain classes of infinite type domains. For results of this nature, see Bergman \cites{Stefan 1933, Bergman 1970}, Hörmander \cite{Hörmander}, Diedrich \cites{Diederich 1970, Diederich 1973}, Fefferman \cite{Fefferman 1974}, and Kamimoto \cite{Kamimoto}. In a recent article, we \cite{Ravi} prove asymptotic limits of the Bergman kernel and metric at exponentially flat infinite type boundary points of domains in $\mathbb{C}^n$.

Upper bound estimates of the Szeg\H{o} kernel on bounded smooth pseudoconvex domains of finite type in $\mathbb{C}^2$ were established by Nagel, Rosay, Stein, and Wainger \cite{Nagel}, and on bounded convex domains of finite type in $\mathbb{C}^n$ by McNeal \cite{McNeal 1997}. Wu and Xing \cite{Wu 2022} obtained the lower bound of the Szeg\H{o} kernel on the diagonal for any bounded smooth pseudoconvex domain in $\mathbb{C}^2$, although their lower bounds may not be sharp for some classes of domains.

The article is organized as follows. We recall the definitions and some of the properties of the Bergman kernel, metric and the Szeg\H{o} kernel in Section \ref{Prelim}. We prove some technical lemmas in Section \ref{Technical Lemmas}, which are used in the proof of the main theorems.
Section \ref{Main Theorem} is devoted to stating and proving the main results of this article. We shall
use $A \lesssim B$ to mean $A \leq C B$ for some constant $C > 0$, 
independent of certain parameters that will be clear from the context.
\section{Preliminaries}\label{Prelim}
\subsection{Bergman kernel and metric}
Let $D \subset \mathbb{C}^n$ be a bounded domain. The Bergman space of $D$, denoted by $A^2(D)$, consists of holomorphic functions $f:D\to\mathbb{C}$ that are also in $L^2(D)$. Since $A^2(D)$ is a closed subspace of $L^2(D)$, there exists an orthogonal projection of $L^2(D)$ onto $A^2(D)$, called the Bergman projection of $D$. The Bergman projection, denoted by $P_{D}$, is given by an integral kernel, called the Bergman kernel and denoted by $K_D: D\times D\to\mathbb{C}$, i.e.
\begin{align}
    P_{D}f(z) = \int_{D}K_{D}(z, w)f(w) \, \mathrm{d}V(w) \quad\text{for each } f \in L^2(D), z \in D.
\end{align}

The Bergman kernel satisfies the following properties: 
\begin{enumerate}
    \item For a fixed $w\in D$, $K_D(\cdot, w)\in A^2(D)$,
    \item $\overline{K_D(z, w)} = K_D(w,z)$ for each $z,w\in D$, and
    \item $f(z)=\int_D K_D(z, w)f(w)\,\mathrm{d}V(w)$ for each $f\in A^2(D)$ and $z\in D$.
\end{enumerate} 

The Bergman kernel on the diagonal of $D$ is denoted by 
\begin{align}
  \kappa_D(z) = K_D(z,z).  
\end{align}
Moreover, the Bergman metric of $D$ is defined by
\begin{align} 
B_D(z;\xi)=\sqrt{\sum_{i,j=1}^{n+1}g_{i\bar{j}}(z)\xi_i\overline{\xi}_j},
\end{align}
where 
$g_{i\bar{j}}(z) = \partial^2_{z_i\bar{z}_j} \operatorname{log} \kappa_D(z)$, $z \in D$, and $\xi = (\xi_1, \xi_2, \dots, \xi_{n+1}) \in \mathbb{C}^{n+1}$.

The following extremal integrals are useful in computing and estimating $k_D$, and $B_D$. For $z \in D$, $\xi \in \mathbb{C}^{n + 1} \setminus \{0\}$, define
\begin{align}
    I_0^D(z) &= \operatorname{inf}\left\{\int_D|f|^2 \, \mathrm{d}V : f \in A^2(D), f(z) = 1\right\}, \text{ and}\\
    I_1^D(z, \xi) &= \operatorname{inf}\left\{\int_D|f|^2 \, \mathrm{d}V: f \in A^2(D), f(z) = 0, \sum_{j= 1}^{n+1}\xi_j\frac{\partial f}{\partial z_j}(z) = 1\right\}.
\end{align}
It is easy to see from the definitions that if $D'$ is a subdomain of $D$ and $z \in D'$, then
\[I_0^{D'}(z) \leq I_0^{D}(z), \text{ and}\quad I_1^{D'}(z, \xi) \leq I_1^{D}(z, \xi).\]

If $f: D_1 \to D_2$ is a biholomorphism. Then the following transformation formulae hold.
\begin{align}
        I_0^{D_1}(z)|\operatorname{det}J_{\mathbb{C}}f(z)|^2 = I_0^{D_2}(f(z)), \text{ and}\quad
        I_1^{D_1}(z, \xi)|\operatorname{det}J_{\mathbb{C}}f(z)|^2 = I_1^{D_2}(f(z), J_{\mathbb{C}}f(z) \xi),
\end{align}
where $J_{\mathbb{C}}f(p)$ denotes the complex Jacobian matrix of $f$ at $p$. 

The following proposition gives a useful representation of the Bergman kernel on the diagonal, and the Bergman metric in terms of the above extremal integrals.
\begin{prop}[{Bergman \cite{Stefan 1933}}]\label{Fuchs}
    Let $D \subset \mathbb{C}^{n + 1}$ be a domain. Then, for $z \in D$
    and $\xi \in \mathbb{C}^{n + 1} \setminus \{0\}$
\begin{align*}
        \kappa_D(z) &= \frac{1}{I_0^D(z)}, \text{ and} \quad
        B_D^2(z; \xi) = \frac{I_0^D(z)}{I_1^D(z; \xi)}.
\end{align*}
\end{prop}
We will use the following localization result of the Bergman kernel and metric to prove the lower and upped bounds of the Bergman kernel and metric.
\begin{thm}[Diederich, Fornaess, and Herbort \cite{Fornaess 1984}]\label{2.2}
Let $D \subset \mathbb{C}^n$ be a bounded pseudoconvex domain and $V \subset \subset U$ be 
    open neighbourhoods of a point $z_0 \in bD$. Then there exists a constant $C \geq 1$ such that 
    \begin{align*}
            &\frac{1}{C} \kappa_{D \cap U}(z) \leq \kappa_{D}(z) \leq \kappa_{D \cap U}(z),\\
            \frac{1}{C}B_{D \cap U}&(z;X) \leq B_{D}(z;X) \leq CB_{D \cap U}(z; X)  
    \end{align*}
for  any $z \in D \cap V$ and any $X \in \mathbb{C}^n$.
\end{thm}
\subsection{Szeg\H{o} kernel}
Let $D$ be a bounded smooth domain in $\mathbb{C}^{n}$. For $\epsilon > 0$, let $D_{\epsilon} = \{z \in D: \delta_{D}(z) > \epsilon\}$, where $\delta_{D}(z)$ denotes the Euclidean distance from $z$ to the boundary $bD$. For $1 < p < \infty$, the Hardy space $H^p(D)$ is the space of holomorphic functions $f$ on $D$ such that
\begin{align}
    \norm{f}_{H^p(D)}^p = \limsup_{\epsilon \to 0^+} \int_{bD_{\epsilon}}|f(z)|^p \, \mathrm{d}\sigma(z) < \infty.
\end{align}
The level sets $bD_{\epsilon}$ in the above definition can be replaced by those of any defining function of $D$ (see Stein \cite{Stein 1972}). A classical result
(see \cite{Stein 1972}*{Theorem $10$}) says that if $f \in H^p(D)$, the non-tangential limit of $f$, denoted by $f^{\star}$, exists for almost every point on $bD$.
Furthermore, we have $f^{\star}\in L^p(bD), \, \norm{f}_{H^p(D)} = \norm{f^{\star}}_{L^p(bD)}$, and
\begin{equation}
    f(x) = \int_{bD}P(x, y)f^{\star}(y)\, \mathrm{d}\sigma(y),
\end{equation}
where $P(x,y)$ is the Poisson kenel of $D$.

The Szeg\H{o} kernel is the reproducing kernel of $H^2(D)$, i.e.,
for each $f \in H^2(D)$, we have
\begin{align}
    f(z) = \int_{bD}f^{\star}(\xi)S_{D}(z,\xi) \, \mathrm{d}\sigma(\xi), \, \, \, \forall z \in D.
\end{align}
The Szeg\H{o} kernel on the diagonal of $D$ is denoted by
\begin{align}
    \mathcal{S}_D(z) = S_{D}(z, z).
\end{align}
It follows from the reproducing property of the Szeg\H{o} kernel that
\begin{align}
    \mathcal{S}_{D}(z) =\operatorname{sup} \left\{|f(z)|^2 : f \in H^2(D), \norm{f}_{H^2(D)} \leq 1\right\}.
\end{align}
It is not known whether the Szeg\H{o} kernel on the diagonal is monotone with respect to the domains, but we have the following substitute due to Chen and Fu \cite{Chen2011}.
\begin{lem}[{\cite{Chen 2011}*{Lemma $2.3$}}]\label{Chen lemma}
    Let $D_1 \subset D_2$ be bounded domains in $\mathbb{C}^n$ with $C^2$-smooth boundaries. Then there exists a constant $C > 0$ such that
    \begin{align}
        \mathcal{S}_{D_2}(z) \leq C \mathcal{S}_{D_1}(z)
    \end{align}
    for each $z \in D_1.$
\end{lem}
We use the above lemma and the following theorem of Chen and Fu \cite{Chen 2011} to prove the upper bound of the Szeg\H{o} kernel on the diagonal. In this theorem, Chen and Fu \cite{Chen 2011} 
estimate the Szeg\H{o} kernel in terms of the Bergman kernel and distance to the boundary function $\delta_D$. 
\begin{thm}[{\cite{Chen 2011}*{Theorem $1.3$}}]\label{Chen theorem}
    Let $D \subset \subset \mathbb{C}^n$ be a bounded convex domain with $C^2$ boundary. Then there exists positive constants $C_1$ and $C_2$ such that
    \begin{align}
        C_1\delta_D(z) \leq \mathcal{S}_D(z)/\kappa_D(z) \leq C_2 \delta_D(z).
    \end{align}
\end{thm}
\subsection{\texorpdfstring{$\overline{\partial}_b$} 
--Operator}
Consider a bounded domain 
$D \subset \mathbb{C}^n$ 
with smooth boundary, and let $\rho$ be a smooth
defining function for $D$. In this context, we introduce the
tangential Cauchy–Riemann operator 
$\overline{\partial}_b$ following the notation in Chen and Shaw \cite{Chen}*{Theorem $1.3$}.

Let $\Lambda_{p,q}(\mathbb{C}^n)$ denote the set of smooth $(p,q)$ forms on $\mathbb{C}^n$. Define $I_{p,q} = \cup_{z \in U}{I^z_{p, q}}$ in a neighbourhood $U$ of $bD$, where each element in the fiber $I_{p, q}^z, \, z \in U$, can be expressed as
\begin{align}
    \rho u_1 + \overline{\partial}\rho \wedge u_2,
\end{align}
where $u_1 \in \Lambda_{p, q}(\mathbb{C}^n)$ and $u_2 \in \Lambda_{p, q - 1}(\mathbb{C}^n)$.
Let $\Lambda_{p, q}(bD)$ be
the orthogonal complement of $I_{p, q}|_{bD}$ in $\Lambda_{p, q}(\mathbb{C}^n)|_{bD}$. Let $\varepsilon_{p,q}(bD)$ be the space of smooth sections of 
$\Lambda_{p, q}(bD)$ over $bD$, essentially capturing $(p, q)$ forms on $bD$. The map $\tau : \Lambda_{p, q}(\mathbb{C}^n) \to \Lambda_{p, q}(bD)$ is defined by restricting a $(p,q)$ form $u$ to $bD$ and then projecting the restriction to $\Lambda_{p, q}(bD)$ when $q>0$. For $q=0$, $\tau$ is the standard restriction map. For $u \in \varepsilon_{p, q}(bD)$, define $\overline{\partial}_b u$ by choosing $\widetilde{u} \in \Lambda_{p, q}(\mathbb{C}^n)$ such that $\tau \widetilde{u} = u$, and set $\overline{\partial}_b u$ to be $\tau \overline{\partial}\widetilde{u}$. The independence of the definition of $\overline{\partial}_b$ on the choice of $\widetilde{u}$ is evident.

Consider the spaces $L^2(bD)$ and $L_{p, q}^2(bD)$, where $L^2(bD)$ consists of square-integrable functions on $bD$ with respect to the induced metric from $\mathbb{C}^n$ to $bD$, and $L_{p, q}^2(bD)$ consists of $(p, q)$ forms on $bD$ with $L^2$ coefficients.

The operator $\overline{\partial}_b$ can be extended to a linear, closed, densely defined operator
\begin{align*}
    \overline{\partial}_b : L_{p, q}^2(bD) \to L_{p, q + 1}^2(bD).
\end{align*}
A function $f \in L^2(bD)$ is said to be a CR-function if $\overline{\partial}_b f = 0$ holds in the sense of distributions. 

The following Theorem and Lemma of Chen and Shaw \cite{Chen} are useful in proving a localization result for the Szeg\H{o} kernel (see Theorem \ref{locSzegö}).
\begin{thm}[{\cite{Chen}*{Theorem $9.2.5$}}]\label{2.5} Let $D$ be a bounded pseudoconvex domain in $\mathbb{C}^n$ with smooth boundary. For and $\alpha \in \varepsilon_{p, q}(bD)$, where $0 \leq p \leq n$ and $1 \leq q \leq n - 1$, the following conditions are equivalent:
\begin{enumerate}
    \item There exists $u \in \varepsilon_{p, q - 1}$ satisfying $\overline{\partial}_b u = \alpha$ on $bD$.
    \item There exists $\Tilde{\alpha} \in C^{\infty}_{p, q}(\overline{D})$ with $\tau \Tilde{\alpha} = \alpha \, (\text{or } \Tilde{\alpha} = \alpha \text{ on } bD)$ and $\overline{\partial}\Tilde{\alpha} = 0$ in $D$.
    \item $\int_{bD} \alpha \wedge \psi = 0$,$\quad \phi \in \varepsilon_{n - p, n - q -1}(bD) \cap \text{Ker}\left( \overline{\partial}_b\right)$.
\end{enumerate}   
When $1 \leq q < n - 1$, the above conditions are equivalent to 
\begin{align*}
    \overline{\partial}_b\alpha = 0 \text{ on } bD.
\end{align*}
\end{thm}
\begin{lem}[{\cite{Chen}*{Lemma $9.3.7$}}]\label{2.6}
    Let $D$ be a bounded pseudoconvex domain in $\mathbb{C}^n$ with smooth boundary. Let $\alpha \in C^{\infty}_{p, q}(bD)$, where $0 \leq p \leq n, \, 1 \leq q \leq n - 2$, such that $\overline{\partial}_b{\alpha} = 0$ on $bD$. For every nonnegative integer $s$, there exists $u_s \in W^s_{p, q - 1}(bD)$ satisfying $\overline{\partial}_b u_s = \alpha$ on $bD$. Furthermore, there exists a constant $C_s$ independent of $\alpha$ such that 
    \begin{align*}
        \norm{u_s}_{s(bD)} \leq C_s \norm{\alpha}_{s(bD)}.
    \end{align*}
    When $q = n - 1, \, \alpha \in C^{\infty}_{p, n - 1}(bD)$ and $\alpha$ satisfies 
    \begin{align*}
        \int_{bD} \alpha \wedge \phi = 0, \quad \phi \in C_{n -  p, 0}^{\infty}(bD) \cap \operatorname{Ker}\left(\overline{\partial}_b\right),
    \end{align*}
    the same conclusion holds.
\end{lem}
\subsection{Finite and Mildly Infinite Type}
In this subsection, we give the definitions of functions falling under two categories: finite type at zero and mildly infinite type at zero. 
\begin{defn}\label{finite type}
    A function \(f : [0,\infty) \to [0, \infty)\) is said to be \emph{finite type at zero} if it satisfies the following conditions:
    \begin{enumerate}
        \item \(f\) is continuous and strictly increasing on the interval \([0, \infty)\), and
        \item There exist \(m > 1\) and \(c > 0\) such that
        \[\lim_{x \to 0^+} \frac{f(x)}{x^m} = c > 0.\]
    \end{enumerate}
\end{defn}
It is easy to see from the definition that if $f$ is finite type at zero, then $f(0) = 0$ and $f(x) > 0$ for each $x \in (0, \infty)$.
\begin{exmp}
    For $m > 1$, $f(x) = x^{m}$ is finite type at zero.
\end{exmp}
\begin{defn}
An increasing function $g : [0, R] \to \mathbb{R}$ is said to satisfy a \emph{doubling
condition} if $g(0) = 0$ and there exists a constant $\sigma > 1$ such that
\[2g(x) \leq g(\sigma x) \text{ for each } x \in [0, R/\sigma].\]
We will call the constant $\sigma > 1$ a \emph{doubling constant} for $g$.    
\end{defn}
For a strictly increasing function $f : [0, \infty) \to [0, \infty)$, define
    \[ \Lambda_f(x) := \begin{cases} 
      -1/ \operatorname{log}(f(x)), &  \text{if } 0 < x < f^{-1}(1), \text{ and}\\
      0, & \text{if } x = 0.
    \end{cases}
    \]
\begin{defn}\label{infinite type}
    A function $f : [0, \infty) \to [0, \infty)$ is said to be \emph{mildly infinite type at zero} if it satisfies the following conditions:
    \begin{enumerate}
        \item $f$ is smooth and strictly increasing on the interval $[0, \infty )$,
    \item $f^k(0) = 0$, for each $k \in \mathbb{N} \cup \{0\}$, and
    \item 
    $\Lambda_f|_{[0, R]}$ satisfies the doubling condition for some $R \in \left(0, f^{-1}(1)\right)$. 
    \end{enumerate}
\end{defn}
\begin{exmp}
    For $p > 0$, $f(x) = e^{-1/x^{p}}$ is mildly infinite type at zero.
\end{exmp}
\section{Technical Lemmas}\label{Technical Lemmas}
We prove the following lemma, which serves as a valuable tool in proving Theorem \ref{finite type lemma}.
\begin{lem}\label{Lambda_f}
    Let $f: [0, \infty) \to [0, \infty)$ be a strictly increasing functon satisfying $f(0) = 0$ and $m \geq 1$. Assume $\Lambda_f|_{[0, R]}$
    satisfies a doubling condition for some $R \in (0, f^{-1}(1))$. Then there exist constants $T >0$ and $C > 0$ such that
    \begin{align}
        [\Lambda_f^{-1}(2t)]^{m} - [\Lambda_f^{-1}(t)]^m \leq C[\Lambda_f^{-1}(t)]^m, 
    \end{align} for each $0 < t < T$.
\end{lem}
\begin{proof}
Let $\sigma > 1$ be a doubling constant for $\Lambda_f$. Write $T := 
\Lambda_f(R/\sigma)$. By the doubling condition of $\Lambda_f$, we have
\[2\Lambda_f(x) \leq \Lambda_f(\sigma x),\]
for each $x \in [0, R/\sigma]$. Since $\Lambda_f^{-1}$ is strictly increasing, we get
\[\Lambda_f^{-1}(2 \Lambda_f(x)) \leq \sigma x,\]
for each $x \in [0, R/\sigma].$
Let $t \in [0, T],$ then $\Lambda_f^{-1}(t) \in [0, R/\sigma]$. Therefore,
\[\Lambda_f^{-1}(2t) \leq \sigma \Lambda_f^{-1}(t).\]
Hence,
\[[\Lambda_f^{-1}(2t)]^m \leq \sigma^m [\Lambda_f^{-1}(t)]^m.\]
Now take $C = \sigma^m - 1 > 0$, then we get
\[[\Lambda_f^{-1}(2t)]^{m} - [\Lambda_f^{-1}(t)]^m \leq C[\Lambda_f^{-1}(t)]^m, \]
for each $t \in [0, T]$.
\end{proof}
We now prove the following proposition which plays a crucial role in proving our main results.
\begin{prop}\label{finite type lemma} Let \(f : [0, \infty) \to [0, \infty)\) be either finite type at zero or mildly infinite type at zero. Then, for each $k \geq 1$ and $M > 0$, there exists constant $C(k, M, f) > 0$ such that
\begin{align}\label{Integral estimate}
   \int_{0}^{1}\frac{r^{k-1}}{(f(r) + t)^k} \, \mathrm{d}r \leq C(k, M, f) \frac{[f^{-1}(t)]^k}{t^k},
\end{align}
for each $ 0 < t < M $.
\end{prop}
\begin{proof}
The proof is divided into two cases: $(1)$ $f$ is finite type at zero, and $(2)$ $f$ is mildly infinite type at zero.

\textbf{Case 1.} Let $f$ be finite type at zero.
Then, there exist $m > 1$, $0 < c_1 < 1$ and $c_2 > 0$ such that
\begin{align*}
    c_1 \leq \frac{f(x)}{x^m} \leq c_2,
\end{align*}
for each $x \in (0, \operatorname{max}\{1, f^{-1}(M)\})$.
Hence,
\begin{align*}
        \int_{0}^{1} \frac{r^{k-1}}{(f(r) + t)^k} \mathrm{d}r &\leq \frac{1}{c_{1}^{k}} \int_{0}^{1} \frac{r^{k-1}}{(r^m + t)^k} \mathrm{d}r.
\end{align*}
By using the change of variable $r = (t \operatorname{tan}^2(\theta))^{1/m}$, we get   
\begin{align*}
    \int_{0}^{1} \frac{r^{k-1}}{(r^m + t)^k} \mathrm{d}r &\leq \frac{2}{m}\int_{0}^{\operatorname{tan}^{-1}(1/\sqrt{t})}  \frac{(t \operatorname{tan}^2(\theta))^{(k-1)/m}}{t^k \operatorname{sec}^{k \beta }(\theta)} \frac{t \operatorname{tan}(\theta) \operatorname{sec}^2(\theta)}{(t \operatorname{tan}^2{(\theta)})^{(m-1)/m}} \, \mathrm{d} \theta \\
    &\leq \frac{2}{m t^k} \int_{0}^{\pi/2} \frac{t^{k/m} \operatorname{tan}^{k \beta /m}(\theta) \operatorname{cos}^{2(k-1)}(\theta)}{\operatorname{tan}^{(1 - 1/m)}(\theta) \operatorname{tan}^{1/m}(\theta)} \, \mathrm{d}\theta\\
    &\lesssim \frac{[f^{-1}(t)]^k}{t^k} \int_{0}^{\pi/2} \frac{\operatorname{sin}^{(k \beta  - 1)/m}(\theta) [\operatorname{cos}(\theta)]^{2(k-1)(1- 1/m)}}{\operatorname{cos}^{1/m}(\theta) \operatorname{tan}^{1-1/m}(\theta)} \, \mathrm{d} \theta \\
    &\lesssim \frac{[f^{-1}(t)]^k}{t^k} \int_{0}^{\pi/2} \frac{1}{\operatorname{cos}^{1/m}(\theta) \operatorname{tan}^{1-1/m}(\theta)} \, \mathrm{d} \theta \\
    &\lesssim \frac{[f^{-1}(t)]^k}{t^k},
\end{align*}
for each $0 < t < M$. 

\textbf{Case 2.} Suppose \(f\) is mildly infinite type at zero, let \(t \in \left(0, \operatorname{min}\{f\left(R/\sigma\right), f(1)^2\}\right)\). Write,
\begin{align*}
     \int_{0}^{1}\frac{r^{k - 1}}{(f(r) + t)^k} \, \mathrm{d} r =& \int_{0}^{f^{-1}(t)}\frac{r^{k - 1}}{(f(r) + t)^k} \, \mathrm{d} r + \int_{f^{-1}(t)}^{f^{-1}(t^{1/2})} \frac{r^{k - 1}}{(f(r) + t)^k} \, \mathrm{d} r + \int_{f^{-1}(t^{1/2})}^{1} \frac{r^{k - 1}}{(f(r) + t)^k} \, \mathrm{d} r \\
     =& I + II + III.
\end{align*}
Now,
\begin{align}
        I &= \int_{0}^{f^{-1}(t)}\frac{r^{k - 1}}{(f(r) + t)^k} \, \mathrm{d} r \leq  \int_{0}^{f^{-1}(t)} \frac{r^{k - 1}}{t^k} \, \mathrm{d}r = \frac{1}{k} \frac{[f^{-1}(t)]^k}{t^k},
\end{align}

\begin{align}
        II &= \int_{f^{-1}(t)}^{f^{-1}(t^{1/2})} \frac{r^{k - 1}}{(f(r) + t)^k} \, \mathrm{d} r
        \leq \frac{1}{t^k} \int_{f^{-1}(t)}^{f^{-1}(t^{1/2})} r^{k - 1} \, \mathrm{d} r \nonumber \\
        &= \frac{1}{kt^k}\left([f^{-1}(t^{1/2})]^k - [f^{-1}(t)]^k\right) \nonumber \\
        &= \frac{1}{kt^k}\left(\left[\Lambda_f^{-1}\left(\frac{2}{\operatorname{log}(1/t)}\right)\right]^k - \left[\Lambda_f^{-1}\left(\frac{1}{\operatorname{log}(1/t)}\right)\right]^k\right) \quad \text{ (Here $f^{-1}(x) = \Lambda_f^{-1}(-1/\operatorname{log}(x))$)} \nonumber \\
        &\leq \frac{C}{kt^k} \left[\Lambda_f^{-1}\left(\frac{1}{\operatorname{log}(1/t)}\right)\right]^k \quad \quad \quad \quad \quad \quad \quad \quad \quad \quad \quad \quad \, \, \,\text{(By Lemma \ref{Lambda_f})} \nonumber \\
        &= \frac{C}{k} \frac{[f^{-1}(t)]^k}{t^k}, \text{ and }
\end{align}
\begin{align*}
    III = \int_{f^{-1}(t^{1/2})}^{1} \frac{r^{k - 1}}{(f(r) + t)^k} \, \mathrm{d} r  \leq \int_{f^{-1}(t^{1/2})}^{1} \frac{r^{k-1}}{k f(r)t^{k - 1}} \, \mathrm{d} r \leq \int_{f^{-1}(t^{1/2})}^{1} \frac{r^{k - 1}}{kt^{1/2}t^{k - 1}} \, \mathrm{d}r \leq \frac{t^{1/2}}{kt^k}.
\end{align*} 
Since
\begin{align*}
    \lim_{x \to 0^+} \frac{f(x)}{x^{2k}} = 0,
\end{align*}
there exists a constant $C(k, f) > 0$ such that
\begin{align}
III \leq C(k,f) \frac{[f^{-1}(t)]^k}{t^k},    
\end{align}
for each $0 < t < \operatorname{min}\{f\left(R/\sigma\right), f(1)^2\}$. Hence from the estimates $I$, $II$, and $III$, we get
\begin{align*}
    \int_{0}^{1} \frac{r^{k-1}}{(f(r) + t)^k} \, \mathrm{d}r \lesssim \frac{[f^{-1}(t)]^k}{t^k},
\end{align*}
for each $ 0 < t < \operatorname{min}\{f\left(R/\sigma\right), f(1)^2\}$.
Therefore there exists constant $C(k, M, f) > 0$ such that
\begin{align}
   \int_{0}^{1}\frac{r^{k-1}}{(f(r) + t)^k} \, \mathrm{d}r \leq C(k, M, f) \frac{[f^{-1}(t)]^k}{t^k},
\end{align}
for each $ 0 < t < M $.
\end{proof}
\section{Proof of the Main Theorems} \label{Main Theorem}
In this section, we prove optimal estimates of the Bergman and Szeg\H{o} kernels on the diagonal, and the Bergman metric near the boundary of smooth bounded generalized decoupled domains in $\mathbb{C}^n$. We also prove lower bounds of the Bergman and Szeg\H{o} kernels on the diagonal near the boundary of more general domains in $\mathbb{C}^n$.
\subsection{Optimal estimates of the Bergman kernel and metric}
We prove optimal estimates of the Bergman kernel on the diagonal below.
\begin{manualtheorem}{1.2}
    Let $D \subset \mathbb{C}^n$ be a bounded smooth pseudoconvex
    generalized decoupled domain near $0 \in bD$.
\begin{enumerate}
    \item Then, there exist constants $C > 0$ and $t_0 > 0$ such that
\begin{align}
    \kappa_{D}(z) \geq C (\operatorname{Re}z_n)^{-2} \Pi_{j = 1}^{k} [f_j^{-1}(\operatorname{Re} z_n)]^{-2n_{j}},
\end{align}
for each $z \in \{w \in \mathbb{C}^n : \operatorname{Re}w_n > F(w'), \text{ and }|w_n| < t_0\}$.
\item Then, for each $\alpha = (\alpha_1, \dots, \alpha_k)$ with ${\alpha_j > 1}$
$(j \in \{1, \dots, k\})$ and $\beta > 1$, there exist constants $C(\alpha, \beta) > 0$ and $t_0 > 0$ such that
\begin{align}
    \kappa_{D}(z) \leq C(\alpha, \beta) (\operatorname{Re}z_n)^{-2} \Pi_{j = 1}^{k} [f_j^{-1}(\operatorname{Re} z_n)]^{-2n_{j}},
\end{align}
for each $z \in \Gamma^{\alpha, \beta}_k \cap \{w \in \mathbb{C}^n : |w_n| < t_0\}$. 
Here
\begin{align}
    \Gamma^{\alpha, \beta}_{k} := \left\{w \in \mathbb{C}^n: \operatorname{Re}w_n > k \beta \left[f_1\left(\alpha_1\left|w^1\right|\right) + \dots + f_k\left(\alpha_k\left|w^k\right|\right)\right]\right\}.
\end{align}
\end{enumerate}
\end{manualtheorem}
\begin{proof}$(1)$
     Since $D \subset \mathbb{C}^n$ is a generalized decoupled domain, there exists a neighbourhood $U$ of the origin such that 
    \begin{align}
        D \cap U = \{z \in U: \operatorname{Re}z_n > F(z')\}.
    \end{align}
    By localization of the Bergman kernel (see Theorem \ref{2.2}), it is enough to compute $\kappa_{D \cap U}$.
    Recall that
    \[\kappa_{D \cap U}(z) = \operatorname{sup}\left\{|f(z)|^2 : f \in A^2(D \cap U), \norm{f}_{A^2(D \cap U)} \leq 1\right\}.\]
    Fix $z \in \{w \in V : \operatorname{Re}w_n > F(w')\}$, where $V \subset \subset U$.
    Let $\phi(w', w_n) = \frac{(\operatorname{Re}z_n)^n}{(w_n - i \operatorname{Im}z_n + \operatorname{Re} z_n)^n}$. Then, $\phi \in A^2(D \cap U)$ and
\begin{align}\label{lower bound 1}
    \kappa_{D \cap U}(z) \geq \frac{|\phi(z)|^2}{\norm{\phi}^2_{A^2({D \cap U)}}} = \frac{1}{2^{2n}} \cdot \frac{1}{\norm{\phi}^2_{A^2(D \cap U)}}.
\end{align}
We prove the lower bound of the Bergman kernel, by finding an upper bound of $\norm{\phi}^{2}_{A^2(D \cap U)}$.

Without loss of generality, we can assume that $U$ is contained in the unit ball centred at zero. Let $w_n = u + iv$. Consider
\begin{align*}
        \int_{D \cap U} |\phi(w', w_n)|^2 \mathrm{d} V(w)
        &\lesssim \int_{|w'| < 1} \int_{F(w')}^{\infty}\int_{-1}^{1} \frac{(\operatorname{Re}z_n)^{2n}}{|u + i (v - \operatorname{Im}z_n) + \operatorname{Re}z_n|^{2n}} \, \mathrm{d}v \,\mathrm{d}u \, \mathrm{d}V(w')\\
        &= \int_{|w'| < 1} \int_{F(w')}^{\infty}\int_{-1}^{1} \frac{(\operatorname{Re}z_n)^{2n}}{(u + \operatorname{Re}z_n)^{2n}\left(\left(\frac{v - \operatorname{Im}(z_n)}{u + \operatorname{Re}z_n}\right)^2 + 1 \right)^n} \, \mathrm{d}v \, \mathrm{d}u \,\mathrm{d}V(w').
\end{align*}
For simplicity of notation, denote $\operatorname{Re}z_n$ by $t$. By using the change of variable $ v = (u + t)s + \operatorname{Im}z_n$, we get
\begin{align}
        \int_{|w'| < 1} \int_{F(w')}^{\infty}&\int_{-1}^{1} \frac{t^{2n}}{(u + t)^{2n}\left(\left(\frac{v - \operatorname{Im}z_n}{u + t}\right)^2 + 1 \right)^n} \, \mathrm{d}v \, \mathrm{d}u \, \mathrm{d}V(w')\nonumber\\  
        &\lesssim \int_{|w'| < 1} \int_{F(w')}^{\infty} \frac{t^{2n}}{(u + t)^{2n - 1}} \, \mathrm{d}u \, \mathrm{d}V(w') \nonumber\\
        &\lesssim \int_{|w'| < 1} \frac{t^{2n}}{(F(w') + t)^{2n - 2}} \, \mathrm{d}V(w')\nonumber\\
        &\leq  \int_{|w^1| < 1} \dots \int_{|w^k| < 1} \frac{t^{2n}}{[f_1(|w^1|) + \dots + f_k(|w^k|) + t]^{ 2(n_1 + n_2 + \dots + n_k)}} \, \mathrm{d}w^{k} \dots \mathrm{d}w^1       \nonumber\\
        &\leq \int_{|w^1| < 1} \dots \int_{|w^k| < 1} \frac{t^{2n}}{(f_1(|w^1|) + t)^{2n_1}  \dots (f_k(|w^k|) + t)^{2n_k}} \, \mathrm{d}w^{k} \dots \mathrm{d}w^1\nonumber\\ 
        &\lesssim t^{2n} \Pi_{j = 1}^{k}\int_{|w^j| < 1} \frac{1}{(f_j(|w^j|) + t)^{2n_j}} \, \mathrm{d}w^j\nonumber\\ &\lesssim t^{2n} \Pi_{j = 1}^{k}\int_{0}^{1} \frac{r^{2n_j - 1}}{(f_j(r) + t)^{2n_j}} \, \mathrm{d}r.
\end{align}
Using Proposition \ref{finite type lemma}, we get
\begin{align}
       t^{2n} \Pi_{j = 1}^{k}\int_{0}^{1} \frac{r^{2n_j - 1}}{(f_j(r) + t)^{2n_j}} \, \mathrm{d}r \lesssim t^{2n} \Pi_{j = 1}^{k} \frac{[f_j^{-1}(t)]^{2n_j}}{t^{2n_j}} \nonumber
       = t^2 \Pi_{j = 1}^{k} [f_j^{-1}(t)]^{2n_j}.
\end{align}

Hence,
\begin{align}\label{estimates of phi}
    \int_{D \cap U} |\phi(w', w_n)|^2 \mathrm{d} V(w) \lesssim (\operatorname{Re}z_n)^2 \Pi_{j = 1}^{k} [f_j^{-1}(\operatorname{Re}z_n)]^{2n_j},
\end{align}
for each $z \in \{w \in V : \operatorname{Re}w_n > F(w')\}$. 
Therefore,
 there exist constants $C > 0$ and sufficiently small $t_0 > 0$ such that
 \begin{align}\label{41}
    \kappa_{D}(z) \geq C (\operatorname{Re}z_n)^{-2} \Pi_{j = 1}^{k} [f_j^{-1}(\operatorname{Re} z_n)]^{-2n_{j}},
\end{align}
for each $z \in \{w \in \mathbb{C}^n : \operatorname{Re}w_n > F(w'), \text{ and } |w_n| < t_0\}$.

$(2)$
There exists $t_0 > 0$ such that for each $z = (z^1, \dots, z^k, z_n) \in \Gamma^{\alpha, \beta}_k \cap \{w \in \mathbb{C}^n : |w_n| < t_0\}$, we show that 
    \begin{align*}
    A_z := B_{n_1}\left(z^1, \frac{\alpha_1 - 1}{\alpha_1}f_1^{-1}\left(\frac{\operatorname{Re}z_n}{k\beta}\right)\right) \times \dots \times B_{n_k}\left(z^k, \frac{\alpha_k - 1}{\alpha_k}f_k^{-1}\left(\frac{\operatorname{Re}z_n}{k\beta}\right)\right) \times\\ B_1\left(z_n, \frac{\beta - 1}{\beta}\operatorname{Re}z_n\right)
    \end{align*}
is contained in $U$.
Let $z \in \Gamma^{\alpha, \beta}_k \cap \{w \in \mathbb{C}^n : |w_n| < t_0\}$. We want to prove $A_z \subset D \cap U$.
For $w = (w^1, \dots, w^k, w_n) \in A_z$, we have
\begin{align}\label{etaestimate}
    |w^j| \leq |z^j| + \frac{(\alpha_j - 1)}{\alpha_j}f_j^{-1}\left(\frac{\operatorname{Re}z_n}{k\beta}\right),
\end{align}
for each $j \in \{1, 2, \dots, k\}$.
Since $ z = (z^1, \dots, z^k, z_n) \in \Gamma^{\alpha, \beta}_{k}$, we have
\begin{align}
        \operatorname{Re}z_n &>  k \beta (f_1(\alpha_1|z^1|) + \dots + f_k(\alpha_k|z^k|)) \geq k\beta f_j(\alpha_{j}|z^j|).
\end{align}
Therefore,
\begin{align}\label{xiestimate}
    |z^j| < \frac{f_j^{-1}(\operatorname{Re}z_n/k\beta)}{\alpha_{j}},
\end{align}
for each $j$. Using \eqref{etaestimate} and \eqref{xiestimate}, we get
\begin{align}
    \left|w^j\right| < f_j^{-1}\left(\frac{\operatorname{Re}z_n}{k\beta}\right) \text{ and hence } f_j\left(\left|w^j\right|\right)< \frac{\operatorname{Re}z_n}{k\beta},
\end{align}
for each $j = 1, 2, \dots, k$. Therefore,
\begin{align}\label{f_jestimate}
    f_1\left(\left|w^1\right|\right)+\dots+f_k\left(\left|w^k\right|\right) < \frac{\operatorname{Re}z_n}{\beta}.
\end{align}
Since $w_n \in B_1\left(z_n, \frac{\beta - 1}{\beta}\operatorname{Re}z_n \right)$, $\operatorname{Re}w_n > \operatorname{Re}z_n/\beta$. Using \eqref{f_jestimate}, we get
\begin{align}
\operatorname{Re}w_n > f_1\left(\left|w^1\right|\right) +  \dots + f_k\left(\left|w^k\right|\right). 
\end{align}
Therefore, $A_z \subset D \cap U$ for each $ z \in \Gamma^{\alpha, \beta}_k \cap \{w \in \mathbb{C}^n : |w_n| < t_0\}$. Using the monotonicity of the Bergman kernel on the diagonal, we get
\begin{align}\label{kestimate}
    \kappa_{D \cap U}(z) \leq  \kappa_{A_z}(z),
\end{align}
for each $z \in \Gamma^{\alpha, \beta}_k \cap \{w \in \mathbb{C}^n : |w_n| < t_0\}$.
Since 
\begin{align}\label{20}
        \kappa_{A_z}(z) &= \kappa_{A_z}\left(z^1, \dots, z^k, z_n\right) = \frac{1}{\operatorname{Vol}(A_z)} \leq C(\alpha, \beta) \frac{1}{(\operatorname{Re}z_n)^2 \Pi_{j=1}^{k}(f_j^{-1}(\operatorname{Re}z_n/k \beta ))^{2n_j}},
\end{align}
for each $z \in \Gamma^{\alpha, \beta}_k \cap \{w \in \mathbb{C}^n : |w_n| < t_0\}.$

If $f_j$ is finite type at zero, then there exists a constant $C > 0$ such that
\begin{align}\label{21}
    f_j^{-1}(t) \leq Cf_j^{-1}(t/ k \beta ),
\end{align}
for each $t \in (0, 1)$.

If $f_j$ is mildly infinite type at zero, then
\begin{align}\label{22}
        f_j^{-1}\left(\frac{t}{k \beta }\right) &\geq f_j^{-1}\left(t^2\right) \quad \quad \quad \quad \text{ ($f_j$ is strictly increasing and $t \in (0, 1/k \beta )$)}\nonumber\\
        &= \Lambda_f^{-1}\left(\frac{1}{\operatorname{log}(1/t^2)}\right) \, \text{ (Here $f^{-1}(x) = \Lambda_f^{-1}(-1/\operatorname{log}(x))$)}\nonumber\\
        &= \Lambda_f^{-1}\left(\frac{1}{2 \operatorname{log}(1/t)}\right)\nonumber\\
        &\geq \frac{1}{\sigma}\Lambda_f^{-1}\left(\frac{1}{ \operatorname{log}(1/t)}\right) \text{ (By Lemma \ref{Lambda_f})}\nonumber\\
        &= \frac{1}{\sigma}f_j^{-1}(t),
\end{align}
for each $t \in (0,t_0)$, where $t_0 > 0$ is
a sufficiently small number.

Using \eqref{kestimate}, \eqref{20}, \eqref{21}, and \eqref{22}, we get
\begin{align}\label{Upper bound of Bergman kernel}
\kappa_{D}(z) \leq \kappa_{D \cap U}(z) \leq \kappa_{A_z}(z) \leq  C(\alpha, \beta) (\operatorname{Re}z_n)^{-2} \Pi_{j = 1}^{k} [f_j^{-1}(\operatorname{Re} z_n)]^{-2n_{j}},
\end{align}
for each $z \in \Gamma^{\alpha, \beta}_k \cap \{w \in \mathbb{C}^n : |w_n| < t_0\}.$
\end{proof}
We now prove optimal estimates of the Bergman metric.
\begin{manualtheorem}{1.3}
Let $D \subset \mathbb{C}^n$ be a bounded smooth pseudoconvex
generalized decoupled domain near $0 \in bD$.
Then, for each $\alpha = (\alpha_1, \dots, \alpha_k)$ with $\alpha_j > 1$ $(j \in \{1, 2, \dots, k\})$ and $\beta > 1$, there exist constants $C(\alpha, \beta) > 0$ and $t_0 > 0$ such that 
\begin{align*}
    \frac{1}{C(\alpha, \beta)}\left(\sum_{j = 1}^{k}\frac{|\xi^j|^2}{(f_j^{-1}(\operatorname{Re}z_n)^2} + \frac{|\xi_n|^2}{(\operatorname{Re}z_n)^2}\right)\leq B^2_D(z; \xi) \leq C(\alpha, \beta) \left(\sum_{j = 1}^{k}\frac{|\xi^j|^2}{(f_j^{-1}(\operatorname{Re}z_n)^2} + \frac{|\xi_n|^2}{(\operatorname{Re}z_n)^2}\right)
\end{align*}
for each $z \in \Gamma^{\alpha, \beta}_{k} \cap \{w \in \mathbb{C}^n: |w_n| < t_0\}$, and $\xi = (\xi^1, \dots, \xi^k, \xi_n) \in \mathbb{C}^{n_1} \times \dots \times \mathbb{C}^{n_k} \times \mathbb{C}$. 
\end{manualtheorem}
\begin{proof}
 Since $D \subset \mathbb{C}^n$ is a generalized decoupled domain, there exists a neighbourhood $U$ of the origin such that 
    \begin{align}
        D \cap U = \{z \in U: \operatorname{Re}z_n > F(z')\}.
    \end{align}
    Without loss of generality, we can assume that $U$ is contained in the unit ball centred at zero. 
    We divide the proof into two steps. In the first step, we prove the lower bound of the Bergman metric and in the second step, we prove the upper bound of the Bergman metric.

\textbf{Step 1.} By using the localization Theorem \ref{2.2}, for $V \subset \subset U$, there exists $C > 0$ such that
\[\frac{1}{C}B_{D \cap U}^2(z;\xi) \leq B_{D}^2(z;\xi) \leq CB_{D \cap U}^2(z; \xi),\]
for any $z \in D \cap V$ and any $\xi \in \mathbb{C}^n$.
By Proposition \ref{Fuchs}, we have 
\begin{align}
    B_{D \cap U}^2(z;\xi) &= \frac{1}{\kappa_{D \cap U}(z) \cdot I_1^{D \cap U}(z;\xi)}
\end{align}
From Theorem \ref{Bergman kernel}, there exist constants $C(\alpha, \beta) > 0$ and $t_0 > 0$ such that
\begin{align}
    B_{D \cap U}^2(z; \xi)
    \geq \frac{ (\operatorname{Re}z_n)^{2} \Pi_{j = 1}^{k} [f_j^{-1}(\operatorname{Re} z_n)]^{2n_{j}}}{C(\alpha, \beta) \cdot I_1^{D \cap U}(z;\xi)}  
\end{align}
for each $z \in \Gamma^{\alpha, \beta}_k \cap \{w\in \mathbb{C}^n: |w_n| < t_0\}$, and $\xi \in \mathbb{C}^{n}\setminus\{0\}$.
We estimate the upper bound for $I_1^{D \cap U}$. Fix $z \in \Gamma^{\alpha, \beta}_k \cap \{w \in \mathbb{C}^n:|w_n| < t_0\}$. First suppose $\xi_n \neq 0$. Let
\[\psi(w', w_n) = \frac{(2\operatorname{Re}z_n)^{n + 1}(w_n - i \operatorname{Im}z_n - \operatorname{Re}z_n)}{\xi_n (w_n - i\operatorname{Im}z_n + \operatorname{Re}z_n)^{n + 1}}.\]
Then, $\psi \in A^2(D \cap U)$ and
\begin{align}
I_1^{D \cap U}(z, \xi) \leq \norm{\psi}^2_{L^2(D \cap U)} \leq 4 (\norm{\psi_1}^2_{L^2(D \cap U)} + \norm{\psi_2}^2_{L^2(D \cap U)}),
\end{align}
where 
\begin{align*}
    \psi_1(w', w_n) = \frac{(2\operatorname{Re}z_n)^{n + 1}}{\xi_n (w_n - i\operatorname{Im}z_n + \operatorname{Re}z_n)^n } \, \text{  and  }
    \,\psi_2(w', w_n) = \frac{(2\operatorname{Re}z_n)^{n + 2}}{\xi_n (w_n - i\operatorname{Im}z_n + \operatorname{Re}z_n)^{n + 1}}. 
\end{align*}
Here $\psi_1(w', w_n) = \frac{2^{n + 1}\operatorname{Re}z_n}{\xi_n}\phi(w', w_n)$, where $\phi(w', w_n) = \frac{(\operatorname{Re}z_n)^n}{(w_n - i\operatorname{Im}z_n + \operatorname{Re}z_n)^n}$. Hence from \eqref{estimates of phi}, we get 
\begin{align}
    \norm{\psi_1}_{L^2(D \cap U)}^2 &\lesssim \frac{(\operatorname{Re}z_n)^4 \Pi_{j = 1}^{k} [f_j^{-1}(t)]^{2n_j}}{|\xi_n|^2}, \text{ and}\\
    \norm{\psi_2}^2_{L^2(D \cap U)} &\leq 2 \norm{\psi_1}_{L^2(D \cap U)}^2 \lesssim \frac{(\operatorname{Re}z_n)^4 \Pi_{j = 1}^{k} [f_j^{-1}(t)]^{2n_j}}{|\xi_n|^2}.
\end{align}
Therefore,
\begin{align*}
    I_1^{D \cap U}(z, \xi) \lesssim \frac{(\operatorname{Re}z_n)^4 \Pi_{j = 1}^{k} [f_j^{-1}(t)]^{2n_j}}{|\xi_n|^2},
\end{align*}
for each $z \in \Gamma^{\alpha, \beta}_k \cap \{w \in \mathbb{C}^n: |w_n| < t_0\}$ and $\xi \in \mathbb{C}^n$ with $\xi_n \neq 0.$ Hence,
\begin{align}\label{59}
    B^2_{D \cap U}(z ; \xi) \geq \frac{1}{C(\alpha, \beta)} \frac{|\xi_n|^2}{(\operatorname{Re}z_n)^2},
\end{align}
for each $z \in \Gamma^{\alpha, \beta}_k \cap \{w \in \mathbb{C}^n: |w_n| < t_0\}$, and $\xi \in \mathbb{C}^n$. 

Now suppose $\xi^j = (\xi^j_1, \dots, \xi^j_{n_j}) \in \mathbb{C}^{n_j} \setminus \{ 0 \}$ which implies at least one of the $\xi^j_{\ell}$'s are non-zero. 
Without loss of generality, we may assume $\xi^j_{n_j} \neq 0$, and fix $z \in \Gamma^{\alpha, \beta}_k \cap \{w \in \mathbb{C}^n: |w_n| < t_0\}$. Let
\[{\psi}^{\star}(w', w_n) =  \frac{(2\operatorname{Re}z_n)^{n + 1}(w_{n_1 + n_2 + \dots + n_j} - z_{n_1 + n_2 + \dots + n_j})}{\xi^j_{n_j}(w_n - i\operatorname{Im}z_n + \operatorname{Re}z_n)^{n + 1}}.\]
Note that ${\psi}^{\star} \in A^2(D \cap U)$ and
\begin{align}
I_1^{D \cap U}(z, \xi) \leq \norm{\psi^{\star}}^2_{L^2(D \cap U)} \leq 2 (\norm{\psi^{\star}_1}^2_{L^2(D \cap U)} + \norm{\psi^{\star}_2}^2_{L^2(D \cap U)}),
\end{align}
where 
\begin{align}
    \psi^{\star}_1(w', w_n) = \frac{(2\operatorname{Re}z_n)^{n + 1}w_{n_1 + n_2 + \dots + n_j}}{\xi^j_{n_j} (w_n - i\operatorname{Im}z_n + \operatorname{Re}z_n)^{n + 1}}, \text{ and }
    \psi^{\star}_2(w', w_n) = \frac{(2\operatorname{Re}z_n)^{n + 1} z_{n_1 + n_2 + \dots + n_j}}{\xi^j_{n_j} (w_n - i\operatorname{Im}z_n + \operatorname{Re}z_n)^{n + 1}}. 
\end{align}
Since $\psi^{\star}_2(w', w_n) = \frac{2^{n + 1} z_{n_1 + n_2 + \dots + n_j}\operatorname{Re}z_n}{\xi^j_{n_j}(w_n - i\operatorname{Im}z_n + \operatorname{Re}z_n)}\phi(w', w_n)$, where $\phi(w', w_n) = \frac{(\operatorname{Re}z_n)^n}{(w_n - i \operatorname{Im}z_n + \operatorname{Re} z_n)^n}$. Hence from \eqref{estimates of phi}, we get 
\begin{align}
\norm{\psi^{\star}_2}_{L^2(D \cap U)}^2 &\lesssim \frac{|z_{n_1 + \dots + n_j}|^2(\operatorname{Re}z_n)^2 \Pi_{j = 1}^{k} [f_j^{-1}(\operatorname{Re}z_n)]^{2n_j}}{|\xi^j_{n_j}|^2} \nonumber\\
&\lesssim \frac{[f_j^{-1}(\operatorname{Re}z_n)]^2(\operatorname{Re}z_n)^2 \Pi_{j = 1}^{k} [f_j^{-1}(\operatorname{Re}z_n)]^{2n_j}}{|\xi^j_{n_j}|^2},
\end{align}
for each $z \in \Gamma^{\alpha, \beta}_k \cap \{w \in \mathbb{C}^n: |w_n| < t_0\}$. 
Let $w_n = u + iv$. Consider
\begin{align*}
        \norm{\psi_1^{\star}}_{L^2(D \cap U)}^2  & \lesssim \int_{|w'| < 1} \int_{F(w')}^{\infty}\int_{-1}^{1} \frac{(\operatorname{Re}z_n)^{2(n + 1)}|w_{n_1 + n_2 + \dots + n_j}|^2}{|\xi^j_{n_j}|^2|u + i (v - \operatorname{Im}z_n) + \operatorname{Re}z_n|^{2(n + 1)}} \, \mathrm{d}v \,\mathrm{d}u \, \mathrm{d}V(w')\\
        &\lesssim \frac{1}{|\xi^j_{n_j}|^2}\int_{|w'| < 1} \int_{F(w')}^{\infty}\int_{-1}^{1} \frac{(\operatorname{Re}z_n)^{2(n + 1)}|w_{n_1 + n_2 + \dots + n_j}|^2}{(u + \operatorname{Re}z_n)^{2(n + 1)}\left(\left(\frac{v - \operatorname{Im}(z_n)}{u + \operatorname{Re}z_n}\right)^2 + 1 \right)^{n + 1}} \, \mathrm{d}v \, \mathrm{d}u \,\mathrm{d}V(w').
\end{align*}
For simplicity of notation, denote $\operatorname{Re}z_n$ by $t$.
Using the change of variable $ v = (u + t)s + \operatorname{Im}z_n$, we get
\begin{align}
        \int_{|w'| < 1} \int_{F(w')}^{\infty}&\int_{-1}^{1} \frac{t^{2(n + 1)}|w_{n_1 + n_2 + \dots + n_j}|^2}{(u + t)^{2(n + 1)}\left(\left(\frac{v - \operatorname{Im}z_n}{u + t}\right)^2 + 1 \right)^{(n + 1)}} \, \mathrm{d}v \, \mathrm{d}u \, \mathrm{d}V(w') \nonumber\\ 
        &\lesssim \int_{|w'| < 1} \int_{F(w')}^{\infty} \frac{t^{2(n + 1)}|w_{n_1 + n_2 + \dots + n_j}|^2}{(u + t)^{2n + 1}} \, \mathrm{d}u \, \mathrm{d}V(w')\nonumber\\
        &\lesssim \int_{|w'| < 1} \frac{t^{2(n + 1)}|w_{n_1 + n_2 + \dots + n_j}|^2}{(F(w') + t)^{2n}} \, \mathrm{d}V(w')\nonumber\\
        &\leq  \int_{|\eta^1| < 1} \dots \int_{|\eta^k| < 1} \frac{t^{2(n + 1)}|\eta^j|^2}{(f_1(|\eta^1|) + \dots + f_k(|\eta^k|) + t)^{ 2(n_1 + n_2 + \dots + n_k) + 2}} \, \mathrm{d}\eta^k \dots \mathrm{d}\eta^1      \nonumber \\
        &\leq \int_{|\eta^1| < 1} \dots \int_{|\eta^k| < 1} \frac{t^{2(n + 1)}|\eta^j|^2}{(f_1(|\eta^1|) + t)^{2n_1} \dots  (f_k(|\eta^k|) + t)^{2n_k}(f_j(|\eta^j|) + t)^{2}} \, \mathrm{d}\eta^{k} \dots \mathrm{d}\eta^1 \nonumber\\ 
        &\lesssim t^{2(n + 1)} \left( \int_{|\eta^j| < 1} \frac{|\eta^j|^2}{(f_j(|\eta^j|) + t)^{2(n_j + 1)}} \, \mathrm{d}\eta^j \right) \left( \Pi_{\ell = 1, \ell \neq j}^{k}\int_{|\eta^{\ell}| < 1} \frac{1}{(f_{\ell}(|\eta^{\ell}|) + t)^{2n_{\ell}}} \, \mathrm{d}\eta^{\ell} \right)\nonumber\\
        &\lesssim t^{2(n + 1)} \left(\int_{0}^{1} \frac{r^{2n_j + 1}}{(f_j(r) + t)^{2(n_j + 1)}} \, \mathrm{d}r \right)\left(\Pi_{\ell = 1, \ell \neq j}^{k}\int_{0}^{1} \frac{r^{2n_{\ell} - 1}}{(f_{\ell}(r) + t)^{2n_{\ell}}} \, \mathrm{d}r \right)\nonumber\\
        &\lesssim t^{2(n + 1)}\left(\frac{f_j^{-1}(t)}{t}\right)^{2(n_{j} + 1)} \left(\Pi_{\ell = 1, \ell \neq j}^{k} \frac{[f_{\ell}^{-1}(t)]^{2n_{\ell}}}{t^{2n_{\ell}}}\right)\nonumber \text{ (Using Proposition \ref{finite type lemma})}\\
        &= t^2 [f_j^{-1}(t)]^2\Pi_{{\ell} = 1}^{k} [f_{\ell}^{-1}(t)]^{2n_{\ell}}.
\end{align}
Hence,
\begin{align} 
   \norm{\psi_1^{\star}}_{L^2(D \cap U)}^2  \lesssim \frac{t^2 [f_j^{-1}(t)]^2\Pi_{l = 1}^{k} [f_{l}^{-1}(t)]^{2n_{l}}}{|\xi^j_{n_J}|^2}.
\end{align}
Therefore,
\begin{align*}
I_1^{D \cap U}(z, \xi) \lesssim \frac{(\operatorname{Re}z_n)^2 [f_j^{-1}(\operatorname{Re}z_n)]^2\Pi_{j = 1}^{k} [f_j^{-1}(\operatorname{Re}z_n)]^{2n_j}}{|\xi^j_{n_j}|^2},
\end{align*}
for each $z \in \Gamma^{\alpha, \beta}_k \cap \{w \in \mathbb{C}^n: |w_n| < t_0\}$.
Hence,
\begin{align}
    B^2_{D \cap U}(z ; \xi) \geq \frac{1}{C(\alpha, \beta)} \frac{|\xi^{j}_{n_j}|^2}{[f_j^{-1}(\operatorname{Re}z_n)]^2},
\end{align}
for each $z \in \Gamma^{\alpha, \beta}_k \cap \{w \in \mathbb{C}^n: |w_n| < t_0\}$, and $\xi \in \mathbb{C}^n$. Therefore,
\begin{align}\label{66}
    B^2_{D \cap U}(z ; \xi) \geq \frac{1}{n_jC(\alpha, \beta)} \frac{|\xi^{j}|^2}{[f_j^{-1}(\operatorname{Re}z_n)]^2}.
\end{align}
for each $z \in \Gamma^{\alpha, \beta}_k \cap \{(w \in \mathbb{C}^n:|w_n| < t_0\}$, and $\xi \in \mathbb{C}^n$.
After adding the estimates \eqref{59} and \eqref{66}, we get the desired lower bound, i.e.,
\begin{align}
  B^2_D(z; \xi) \geq \frac{1}{C(\alpha, \beta)}\left(\sum_{j = 1}^{k}\frac{|\xi^j|^2}{(f_j^{-1}(\operatorname{Re}z_n)^2} + \frac{|\xi_n|^2}{(\operatorname{Re}z_n)^2}\right),  
\end{align}
for each $z \in \Gamma^{\alpha, \beta}_{k} \cap \{w \in \mathbb{C}^n: |w_n| < t_0\}$, and $\xi = (\xi^1, \dots, \xi^k, \xi_n) \in \mathbb{C}^{n_1} \times \dots \times \mathbb{C}^{n_k} \times \mathbb{C}$.

\textbf{Step 2.} By Proposition \ref{Fuchs}, we have
\begin{align}
    B_{D}^2(z;\xi) &= \frac{1}{\kappa_{D}(z) \cdot I_1^{D}(z;\xi)}\nonumber\\
    &\leq \frac{1}{\kappa_{D}(z) \cdot I_1^{A_{z}}(z; \xi)} \quad \text{ (since $A_z \subset D$)}\nonumber\\
    &= \frac{\kappa_{A_z}(z)}{\kappa_{D}(z)} B_{A_z}^2(z; \xi)\nonumber\\
    &\leq C(\alpha, \beta) \left(\sum_{j = 1}^{k}\frac{|\xi^j|^2}{(f_j^{-1}(\operatorname{Re}z_n)^2} + \frac{|\xi_n|^2}{(\operatorname{Re}z_n)^2}\right) \text{ (from \eqref{41} and \eqref{Upper bound of Bergman kernel})}
\end{align}
for each $z \in \Gamma^{\alpha, \beta}_k \cap \{w \in \mathbb{C}^n : |w_n| < t_0\}$, and $\xi = (\xi^1, \dots, \xi^k, \xi_n) \in \mathbb{C}^{n_1} \times \dots \times \mathbb{C}^{n_k} \times \mathbb{C}$.
\end{proof}
\subsection{Optimal estimates of the Szeg\H{o} kernel}
Let us first establish a localization result for the Szeg\H{o} kernel (see Theorem \ref{locSzegö}). We will then use this localization result to prove the lower and upper bounds of the Szeg\H{o} kernel. This localization involves an auxiliary Szeg\H{o} kernel which we now define.
\begin{defn}
    Let $D \subset \mathbb{C}^n$ be a bounded smooth domain, and let
    \begin{equation}
    A^{\infty}(D) := \mathcal{O}(D) \cap C^{\infty}(\overline{D}).
\end{equation}
Define the Hardy space $H^2(bD)$ to be the closure, in $L^2(bD)$, of the restrictions of elements of $A^{\infty}(D)$ to $bD$. The Poisson integral formula implies that each
$f \in H^2(bD)$ extends to a holomorphic function $Pf$ in $D$. Furthermore, for each $z \in D$, the map
\begin{align}
    f \mapsto Pf(z)
\end{align}
defines a continuous linear functional on $H^2(bD)$. By the Riesz representation theorem, this linear functional is represented by a kernel $k_z \in H^2(bD)$, which is to say
\begin{align}\label{71}
    Pf(z) = \int_{bD} f(\xi)\overline{k_z(\xi)} \, \mathrm{d}\sigma(\xi).
\end{align}
We define the \emph{Auxiliary Szeg\H{o} kernel} by
\begin{align}
  S^A_D(z, \xi) = \overline{k_z(\xi)}, \quad z \in D, \, \xi \in bD.
\end{align}
A function $f \in H^2(bD)$ may be considered as a function on $\overline{D}$,
where its values in $D$ are given by $Pf$ (see \cite{Krantz book}*{p. $66$}). Similarly, we extend ${S}_D^A(z, \xi)$ to a function on $D \times \overline{D}$ and its values on $D \times bD$ are defined only almost everywhere.

It follows from the reproducing property \eqref{71} that
\begin{align}
    \mathcal{S}^{A}_D(z) = S^A_{D}(z,z) =\operatorname{sup} \left\{|Pf(z)|^2 : f \in H^2(bD), \norm{f}_{L^2(bD)} \leq 1\right\}, \quad z\in D.
\end{align}
\end{defn}
\textbf{Remark.} 
It can be shown that the Szeg\H{o} and auxiliary Szeg\H{o} kernels are the same on convex domains.
We do not know if this is the case on all bounded smooth domains.

We now state and prove a localization result for the Szeg\H{o} kernel on the diagonal.
\begin{thm}\label{locSzegö}
    Let $D \subset \mathbb{C}^n$ be a bounded smooth pseudoconvex generalized decoupled domain near $0 \in bD$. Suppose $D' \subset D$ is an open set such that $bD \cap bD'$ contains a neighbourhood of origin. Then there exist a $\mathbb{C}^n$-neighbourhood $V$ of $0$ and a constant $C > 0$ such that
    \begin{align}
        \frac{1}{C}\mathcal{S}^A_{D'}(z) \leq \mathcal{S}_{D}(z) \leq C \mathcal{S}_{D'}(z),
    \end{align}
for each $z \in D' \cap V$.
\end{thm}
\begin{proof}
Since $bD' \cap bD$ contains a neighbourhood of origin,
there exists $\delta > 0$ such that 
\begin{align}
D' \cap W = D \cap W = \{z \in W : \operatorname{Re}z_n > F(z')\},    
\end{align}
where $W = F^{-1}(-\delta, \delta) \times (-\delta, \delta) \times (-\delta, \delta)$. 

Choose $\chi \in C_c^{\infty}(W)$ such that $\chi = 1$ on $F^{-1}(-\delta/2, \delta/2) \times (-\delta/2, \delta/2) \times (-\delta/2, \delta/2)$ and $0 \leq X \leq 1$.
Let $V = F^{-1}(-\delta/4, \delta/4) \times (-\delta/4, \delta/4) \times (-\delta/4, \delta/4)$. Fix $w \in D' \cap V$, then $\mathcal{S}^A_{D'}(., w) \in H^2(bD')$, which implies there exists a sequence $\left( F_n^w \right)_{n \in \mathbb{N}}$ in $A^{\infty}(D')$ such that
$F_n^w$ converges to $\mathcal{S}^A_{D'}(., w)$ in $L^2(bD')$.
Define, 
\begin{align}
    v^w(z) = \tau\left(\frac{F_n^w(z) \bar \partial \chi(z)}{z_n - w_n}\right) \in C^{\infty}(bD).
\end{align}
Using Theorem \ref{2.5} and Lemma \ref{2.6}, there exists $u^w \in C^{2}(bD)$ such that
\begin{align}\label{78}
    \bar \partial_b u^w = v^w \, \, \text{on} \, \, bD \quad \text{and}
\end{align}
there exists a constant $C > 0$ independent of $v^w$ such that
\begin{align}
    \int_{bD}|u_w|^2 \, \mathrm{d}\sigma &\leq C \int_{b{D}}|v_w|^2 \, \mathrm{d}\sigma \nonumber\\
    &= C\int_{bD' \cap W} \frac{|F^w(z)|^2|\partial \chi(z)|^2}{|z_n - w_n|^2} \, \mathrm{d}\sigma(z)\nonumber\\
    &= C\int_{\{\operatorname{Re}z_n = F(z')\} \cap W \setminus \{|\operatorname{Re}z_n| < \delta/3, |\operatorname{Im}z_n| < \delta/3\}} \frac{|F^w(z)|^2 |\bar \partial \chi(z)|^2}{|z_n - w_n|^2} \, \mathrm{d}\sigma(z) \nonumber\\
    &\lesssim \int_{bD'}|F^w(z)|^2 \mathrm{d}\sigma(z).
\end{align}
Let
\begin{align}
    G_n^w(z) = F_n^w(z)\chi(z) - (z_n - w_n)u^w(z) \in C^2(bD).
\end{align}
Then, by \eqref{78}, $\bar \partial_b G_n^w = 0$ on $D$.
Using \cite{Kytmanov}*{Theorem $7.1$}, we have $G_n^w \in \mathcal{O}(D) \cap C^{2}(\overline{D})$, which implies $G_n^w \in \mathcal{O}(\Omega_F \cap W) \cap C^{2}\left(\overline{\Omega_F \cap W}\right)$, where $\Omega_F = \{z \in \mathbb{C}^n : \operatorname{Re}z_n > F(z')\}$. 
Here $G_n^w = F_n^w$ on $ b\left(\Omega_F \cap W \cap \{z_n = w_n\}\right)$ and $G_n^w(\cdot, w_n), F_n^w(\cdot, w_n) \in \mathcal{O}\left(\Omega_F \cap W \cap \{z_n = w_n\}\right)$. By maximum principle, we get
\begin{align}
    G_n^w(w) = F_n^w(w).
\end{align}
Consider, 
\begin{align}
    \mathcal{S}_{D}(w) \geq \frac{|G_n^w(w)|^2}{\norm{G_n^w}_{L^2(bD)}^2} \geq C \frac{|F_n^w(w)|^2}{\norm{F_n^w}_{L^2(bD')}^2},
\end{align}
for each $n \in \mathbb{N}$ and the constant $C (> 0)$ is independent of $n$ and $w \in D' \cap V$. As $n \to \infty$, we get
\begin{align}\label{lowerboundSzego}
\mathcal{S}_{D}(w) \geq C\frac{\left(\mathcal{S}^A_{D'}(w)\right)^2}{\norm{{S}^A_{D'}(\cdot, 
w)}^2_{L^2(bD')}} = C \frac{\left(\mathcal{S}^A_{D'}(w)\right)^2}{\mathcal{S}^A_{D'}(w)} = C\mathcal{S}^A_{D'}(w),
\end{align}
for each $w \in D' \cap V$. Hence we get the theorem by using \eqref{lowerboundSzego} and Lemma \ref{Chen lemma}.
\end{proof}
We now prove the optimal estimates of the Szeg\H{o} kernel on the diagonal.
\begin{manualtheorem}{1.4}
Let $D \subset \mathbb{C}^n$ be a bounded smooth pseudoconvex generalized decoupled domain near $0 \in bD$.
\begin{enumerate}
    \item Then, there exist constants $C > 0$ and $t_0 > 0$ such that
    \begin{align}
    \mathcal{S}_D(z) \geq C (\operatorname{Re}z_n)^{-1} \Pi_{j = 1}^{k} [f_j^{-1}(\operatorname{Re} z_n)]^{-2n_{j}},
\end{align}
for each $z \in \{w \in \mathbb{C}^n: \operatorname{Re}w_n > F(w'), \text{ and } |w_n| < t_0\}$.
    \item Suppose $bD$ is convex near the origin.  Then, for  each $\alpha = (\alpha_1, \dots, \alpha_k)$ with $\alpha_j > 1$ $(j \in \{1, 2, \dots, k\})$ and $\beta > 1$, there exist constants $C(\alpha, \beta) > 0$ and $t_0 > 0$ such that
\begin{align}
    \mathcal{S}_{D}(z) \leq C({\alpha})(\operatorname{Re}z_n)^{-1} \Pi_{j = 1}^{k} [f_j^{-1}(\operatorname{Re} z_n)]^{-2n_{j}},
\end{align}
for each $z \in \Gamma^{\alpha, \beta}_k \cap \{w \in \mathbb {C}^n :|w_n| < t_0\}$. 
\end{enumerate}
\end{manualtheorem}
\begin{proof}
Since $D \subset \mathbb{C}^n$ is a generalized decoupled domain, there exists a neighbourhood $U$ of the origin such that 
    \begin{align}
        D \cap U = \{z \in U: \operatorname{Re}z_n > F(z')\}.
    \end{align}
$(1)$ There exists a domain $D' \subset \mathbb{C}^n$ with $C^2$ boundary and a neighbourhood $W \subset \subset U$ of origin such that 
\begin{align}
    D' \cap W = D \cap W, \text{ and } D' \subset D \cap U.
\end{align}
Recall,
\[\mathcal{S}^A_{D'}(z) = \operatorname{sup}\{|Pf(z)|^2 : f \in H^2(bD'), \norm{f}_{L^2(bD')} \leq 1\}.\]

Fix $z \in D \cap W$. Let $\phi(w', w_n) = \frac{(\operatorname{Re}z_n)^n}{(w_n - i \operatorname{Im}z_n + \operatorname{Re} z_n)^n}$. Then, $\phi \in H^2(bD')$. By using Theorem \ref{locSzegö}, there exist a neighbourhood  $W_0 \subset \subset W$ of origin and a constant $C > 0$ such that
\begin{align}\label{lower bound 1 Szegö}
    \mathcal{S}_{D}(z) \geq C \mathcal{S}^A_{D'}(z) \geq C\frac{|\phi(z)|^2}{\norm{\phi}^2_{H^2(bD')}} = \frac{C}{2^{2n}} \cdot \frac{1}{\norm{\phi}^2_{H^2(bD')}},
\end{align}
for each $z \in D \cap W_0$.
We want to prove the lower bound of the Szeg\H{o} kernel, therefore it is enough to find the upper bound of $\norm{\phi}^2_{H^2(bD')}.$

Without loss of generality, we can assume that $U$ is contained in the unit ball centred at zero. Let $w_n = u + iv$. Consider
\begin{align*}
        \int_{bD' \cap W} |\phi(w', w_n)|^2 &\mathrm{d} \sigma(w) = \int_{\{w \in W: \operatorname{Re}w_n = F(w')\}} |\phi(w', w_n)|^2 \, \mathrm{d} \sigma(w)\\
        & \lesssim \int_{|w'| < 1} \int_{-1}^{1} \frac{(\operatorname{Re}z_n)^{2n}}{|f_1(|\xi^1|) + \dots + f_k(|\xi^k|) + i (v - \operatorname{Im}z_n) + \operatorname{Re}z_n|^{2n}} \, \mathrm{d}v \, \mathrm{d}V(w').
\end{align*}
For simplicity of notation, denote $\operatorname{Re}z_n$ by $t$. By using the change of variable $ v = (f_1(|\xi^1|) + \dots + f_k(|\xi^k|) + t)s + \operatorname{Im}z_n$, we get
\begin{align}
        \int_{|w'| < 1} \int_{-1}^{1}& \frac{(\operatorname{Re}z_n)^{2n}}{|f_1(|\xi^1|) + \dots + f_k(|\xi^k|) + i (v - \operatorname{Im}z_n) + \operatorname{Re}z_n|^{2n}} \, \mathrm{d}v \, \mathrm{d}V(w') \nonumber \\   
        &\lesssim  \int_{|\xi^1| < 1} \dots \int_{|\xi^k| < 1} \frac{t^{2n}}{(f_1(|\xi^1|) + \dots + f_k(|\xi^k|) + t)^{2n - 1}} \, \mathrm{d}\xi^{k} \dots \mathrm{d}\xi^1 \nonumber \\
        &=  \int_{|\xi^1| < 1} \dots \int_{|\xi^k| < 1} \frac{t^{2n}}{(f_1(|\xi^1|) + \dots + f_k(|\xi^k|) + t)^{1 + 2n_1 + 2n_2 + \dots + 2n_k}} \, \mathrm{d}\xi^k  \dots \mathrm{d}\xi^1     \nonumber  \\
        &\leq \int_{|\xi^1| < 1} \dots \int_{|\xi^k| < 1} \frac{t^{2n}}{t(f_1(|\xi^1|) + t)^{2n_1} \dots (f_k(|\xi^k|) + t)^{2n_k}} \, \mathrm{d}\xi^{k} \dots \mathrm{d}\xi^1 \nonumber \\ 
        &\lesssim t^{2n - 1} \Pi_{j = 1}^{k}\int_{|\xi^j| < 1} \frac{1}{(f_j(|\xi^j|) + t)^{2n_j}} \, \mathrm{d}\xi^j \nonumber \\
        &\lesssim t^{2n - 1} \Pi_{j = 1}^{k}\int_{0}^{1} \frac{r^{2n_j - 1}}{(f_j(r) + t)^{2n_j}} \, \mathrm{d}r
\end{align}
Using Theorem \ref{finite type lemma}, we get
\begin{align}
        t^{2n - 1} \Pi_{j = 1}^{k}\int_{0}^{1} \frac{r^{2n_j - 1}}{(f_j(r) + t)^{2n_j}} \, \mathrm{d}r
        \lesssim t^{2n - 1} \Pi_{j = 1}^{k} \frac{[f_j^{-1}(t)]^{2n_j}}{t^{2n_j}} = t \Pi_{j = 1}^{k} [f_j^{-1}(t)]^{2n_j}.
\end{align}
Hence,
\begin{align}
    \int_{bD' \cap W} |\phi(w', w_n)|^2 \mathrm{d} \sigma(w) \lesssim t \Pi_{j = 1}^{k} [f_j^{-1}(t)]^{2n_j}.
\end{align}
Choose $t_0 > 0$ sufficiently small such that 
\begin{align}
    bD' \cap W^c \subset \{z \in \mathbb{C}^n : |z_n| > 2t_0\}.
\end{align}
Now we have,
\begin{align}
    \int_{bD' \cap W^c} |\phi(w', w_n)|^2 \mathrm{d} \sigma(w) \lesssim t^{2n} \lesssim t^2 \cdot t^{2n_1 + 2n_2 + \dots + 2n_k} \lesssim t \cdot  \Pi_{j = 1}^{k} [f_j^{-1}(t)]^{2n_j}.
\end{align}
Hence,
\begin{align}\label{93}
    \norm{\phi}^2_{H^2(bD')} &= \int_{bD'} |\phi(w', w_n)|^2 \mathrm{d} \sigma(w) \nonumber\\
    &=  \int_{bD'\cap W} |\phi_(w', w_n)|^2 \mathrm{d} \sigma(w) +  \int_{bD' \cap W^c} |\phi(w', w_n)|^2 \mathrm{d} \sigma(w)\nonumber \\
    &\lesssim t \Pi_{j = 1}^{k} [f_j^{-1}(t)]^{2n_j} = (\operatorname{Re}z_n)\Pi_{j = 1}^{k}[f_j^{-1}(\operatorname{Re}z_n)]^{2n_j},
\end{align}
for each $z \in D \cap W_0$.
By using \eqref{lower bound 1 Szegö} and \eqref{93}, there exists $t_0 > 0$ such that
\begin{align}
    \mathcal{S}_{D}(z) \geq C (\operatorname{Re}z_n)^{-1} \Pi_{j = 1}^{k} [f_j^{-1}(\operatorname{Re} z_n)]^{-2n_{j}},
\end{align}
for each $z \in \{w \in \mathbb{C}^n : \operatorname{Re}w_n > F(w'), \text{ and } |w_n|<t_0\}$.

$(2)$ Since $bD \cap U$ is convex, there exists a bounded convex domain $D_1 \subset D\cap U$ with
smooth boundary such that $bD \cap bD_1$ contains a $bD$-neighbourhood of origin. Using Lemma \ref{Chen lemma}, there exists a constant $C > 0$ such that
    \begin{align}
        \mathcal{S}_{D}(z) \leq C \mathcal{S}_{D_1}(z)
    \end{align}
    for each $z \in D_1$. By using Theorem \ref{Chen theorem}, we have
    \begin{align}
        \mathcal{S}_{D_1}(z) \lesssim \delta_{D_1}(z) \kappa_{D_1}(z),
    \end{align}
    for each $z \in D_1$. There exists a neighbourhood $W_0 \subset \subset U$ of origin such that
    $D_1 \cap W_{0} = D_F \cap W_0$, recall $D_F = D \cap U$. Hence from the above estimates of the Szeg\H{o} kernel, we get
    \begin{align}\label{37}
        \mathcal{S}_{D}(z) \lesssim \delta_{D_1}(z) \kappa_{D_1}({z}) \lesssim \operatorname{Re}z_n \kappa_{D_1 \cap W_0}(z),
    \end{align}
    for each $z \in D_1 \cap W_0$. By \eqref{Upper bound of Bergman kernel} there exist constants $C(\alpha, \beta),$ $t_0 > 0$ such that 
    \begin{align}\label{38}
        \kappa_{D_1 \cap W_{0}}(z) \leq C(\alpha, \beta) (\operatorname{Re}z_n)^{-2} \Pi_{j = 1}^{k}[f_j^{-1}(\operatorname{Re}z_n)]^{-2n_j},
    \end{align}for each $z \in \Gamma^{\alpha, \beta}_k \cap \{w \in \mathbb {C}^n : |w_n| < t_0\}$. Hence from \eqref{37} and \eqref{38}, we get
    \begin{align*}
    \mathcal{S}_{D}(z) \leq C_{\alpha}(\operatorname{Re}z_n)^{-1} \Pi_{j = 1}^{k} [f_j^{-1}(\operatorname{Re} z_n)]^{-2n_{j}},
\end{align*}
for each $z \in \Gamma^{\alpha, \beta}_k \cap \{w \in \mathbb {C}^n : |w_n| < t_0 \}$.
\end{proof}
\subsection{Lower bounds of the Bergman and Szeg\H{o} kernels on extended generalized decoupled domains}
Choosing the suitable functions as we did in the proof of the Theorem \ref{Bergman kernel} and Theorem \ref{Theorem 1.3}, we can prove the following theorem related to lower bounds of the Bergman and Szeg\H{o} kernels on the diagonal below. 
\begin{manualtheorem}{1.6}
    Let $D \subset \mathbb{C}^n$ be a bounded $C^2$-smooth domain with $0 \in bD$. Suppose $D$ is an extended generalized decoupled domain near $0$.
    Then there exists constant $C(\delta) > 0$ such that
\begin{align}
    \kappa_D(z) &\geq C(\delta) (\operatorname{Re}z_n)^{-2} \Pi_{j = 1}^{k} [f_j^{-1}(\operatorname{Re} z_n)]^{-n_{j}}, \text{ and} \label{Bergman**}\\
    \mathcal{S}_{D}(z) &\geq C(\delta) (\operatorname{Re}z_n)^{-1} \Pi_{j = 1}^{k} [f_j^{-1}(\operatorname{Re} z_n)]^{-n_{j}},\label{Szegö**}
\end{align}
for each $z \in \{w \in D : |w_n| < \delta/2\}$.
\end{manualtheorem}
\begin{proof}
Recall
\begin{align*}
    \kappa_D(z) &= \operatorname{sup}\{|f(z)|^2: f \in A^2(D), \norm{f}_{L^2(D)} \leq 1\}, \text{and}\\
    \mathcal{S}_{D}(z) &= \operatorname{sup}\{|f(z)|^2 : f \in H^2(D), \norm{f}_{H^2(D)} \leq 1\}.
\end{align*}
Fix $z \in D$. Let $\phi(w', w_n) = \frac{(\operatorname{Re}z_n)^n}{(w_n - i \operatorname{Im}z_n + \operatorname{Re} z_n)^n}$. Then, $\phi \in A^2(D) \cap H^2(D)$ and
\begin{align}
    \kappa_D(z) &\geq \frac{|\phi(z)|^2}{\norm{\phi}^2_{L^2(D)}} = \frac{1}{2^{2n}} \cdot \frac{1}{\norm{\phi}^2_{L^2(D)}}, \text{ and}\label{55}\\
    \mathcal{S}_{D}(z) &\geq \frac{|\phi(z)|^2}{\norm{\phi}^2_{H^2(D)}} = \frac{1}{2^{2n}} \cdot \frac{1}{\norm{\phi}^2_{H^2(D)}}\label{lower bound 1*}.
\end{align}
We want to prove the lower bounds of the Bergman and Szeg\H{o} kernels on the diagonal, it is enough to find the upper bounds of $\norm{\phi}^2_{L^2(D)} \text{ and } \norm{\phi}^2_{H^2(D)}$, respectively.

As we did in the proof of Theorem \ref{Bergman kernel} and Theorem \ref{Theorem 1.3}, we get
\begin{align}
    \norm{\phi}^2_{L^2(D)} &\lesssim (\operatorname{Re}z_n)^{-2} \Pi_{j = 1}^{k} [f_j^{-1}(\operatorname{Re} z_n)]^{-n_{j}}, \text{ and} \label{57}\\
    \norm{\phi}^2_{H^2(D)} &\lesssim (\operatorname{Re}z_n)^{-1} \Pi_{j = 1}^{k} [f_j^{-1}(\operatorname{Re} z_n)]^{-n_{j}},\label{58}
\end{align}
for each $z \in \{w \in D : |w_n| < \delta/2\}$. By using 
the estimates \eqref{55}, \eqref{lower bound 1*}, \eqref{57}, and \eqref{58}, we get lower bounds \eqref{Bergman**} and \eqref{Szegö**} of the Bergman and 
Szeg\H{o} kernels, respectively.
\end{proof}
\section*{Acknowledgements}
I express my gratitude to my advisor, Sivaguru Ravisankar, for numerous insightful discussions, valuable suggestions, and constructive comments that have greatly enriched this work.
\def\MR#1{\relax\ifhmode\unskip\spacefactor3000 \space\fi%
  \href{http://www.ams.org/mathscinet-getitem?mr=#1}{MR#1}}
 
\end{document}